\documentclass[a4paper,11pt,reqno] {amsart}
\usepackage[utf8]{inputenc}
\usepackage[english]{babel}
\usepackage{hyperref}
\usepackage{amsmath}
\usepackage{amsthm}
\usepackage{amsfonts}
\usepackage{amssymb}
\usepackage{enumerate}
\usepackage{color}
\usepackage{graphicx}
\numberwithin{equation}{section}
\allowdisplaybreaks
\theoremstyle{plain}
\newtheorem{theorem}{Theorem}[section]
\newtheorem*{theoremA}{Theorem A}
\newtheorem*{theoremB}{Theorem B}
\newtheorem{proposition}[theorem]{Proposition}
\newtheorem{Lemma}[theorem]{Lemma}

\theoremstyle{definition}
\newtheorem{definition}[theorem]{Definition}
\newtheorem{remark}[theorem]{Remark}
\newtheorem{example}[theorem]{Example}
\makeatletter
\@namedef{subjclassname@2020}{%
\textup{2020} Mathematics Subject Classification}
\makeatother
\newcommand{\C}{\mathbb{C}}
\newcommand{\D}{\mathbb{D}}
\newcommand{\N}{\mathbb{N}}
\newcommand{\R}{\mathbb{R}}

\newcommand{\norm}[1]{\lVert #1\rVert}
\newcommand{\Norm}[1]{\Big\lVert #1\Big\rVert}
\newcommand{\abs}[1]{\lvert #1\rvert}
\newcommand{\Abs}[1]{\left\lvert #1\right\rvert}
\newcommand{\sm}{\setminus}

\newcommand{\tn}[1]{\textnormal{#1}}

\newcommand{\cl}[1]{{#1}^\text{\,\rm cl}}

\newcommand{\diam}{\operatorname{diam}}
\newcommand{\dist}{\operatorname{dist}}

\newcommand{\wt}[1]{\widetilde{#1}}

\newcommand{\siT}{\si(T)}                     

\newcommand{\K}{K}                     
\newcommand{\fK}{f(\K)}
\newcommand{\cE}{E}

\newcommand{\wtK}{\wt\K}
\newcommand{\Ombr}{\Om_{\text{br}}}

\DeclareMathOperator*{\esssup}{ess\,sup}

\newcommand{\eps}{\epsilon}

\newcommand{\Om}{\Omega}
\newcommand{\al}{\alpha}
\newcommand{\be}{\beta}
\newcommand{\de}{\delta}
\newcommand{\la}{\lambda}
\newcommand{\si}{\sigma}

\newcommand{\pt}{\partial}

\newcommand{\Omonebr}{\Omega_{1br}} 
\newcommand{\Omtwobr}{\Omega_{2br}} 
\newcommand{\clOmone}{\cl\Omega_1}

\newcommand{\clOmonebr}{\cl\Omega_{1br}}

\renewcommand{\epsilon}{\varepsilon}
\renewcommand{\phi}{\varphi}

\renewcommand{\ss}{\subset}

\newenvironment{psmallmatrix}
  {\left(\begin{smallmatrix}}
  {\end{smallmatrix}\right)}

\begin{document}

\title[Resolvent estimates for a function of a linear operator]{Resolvent estimates \\ for a function of a linear operator}

\author[G. Bello]{Glenier Bello}
\address{Departamento de Matem\'{a}ticas e
Instituto Universitario de Matem\'{a}ticas y Aplicaciones\\
Universidad de Zaragoza\\
50009 Zaragoza\\
Spain}
\email{gbello@unizar.es}
\author[D. V. Yakubovich]{Dmitry Yakubovich}
\address{
Departamento de Matem\'aticas\\
Universidad Aut\'onoma de Madrid\\
Cantoblanco, 28049 Madrid, Spain}
\email {dmitry.yakubovich@uam.es}
\date{\today}

\subjclass[2020]{47A10, 47A30}
\keywords{resolvent growth, unconditional bases of eigenspaces}

\begin{abstract}
Let $T$ be a bounded linear operator on a Banach space and $f$ an analytic function, defined on the spectrum of $T$. We study the relations between the rate of growth of the resolvent of $T$ and that of $f(T)$. We also discuss 
whether the property of unconditional basisness of eigenspaces or root spaces of $f(T)$ implies the corresponding property for $T$, and related issues. 
\end{abstract}
\thanks{Both authors acknowledge the support of the Spanish Ministry for Science and Innovation under Grant PID2022-137294NB-I00, DGI-FEDER.  G. Bello has also been partially supported by  PID2022-138342NB-I00 for \emph{Functional Analysis Techniques in Approximation Theory and Applications (TAFPAA)}}
\maketitle
\section{Introduction}

Let $T$ be a bounded linear operator on a Banach space $X$. 
As usual, let $\sigma(T)$ denote the spectrum of $T$. 
Given a compact set $K$ containing $\sigma(T)$ and a real number $s\ge1$, 
we say that $T$ has \emph{resolvent growth of order $s$ near $K$} if 
\begin{equation}
	\label{eq:res-growth}
	\norm{(T-\lambda)^{-1}}\le
 C \dist(\lambda,\K)^{-s}, \quad  
	\lambda\in D(a,R)\sm\K,  
\end{equation}
where $C$ is a constant and 
$D(a,R)=\{\la\in\C: |\la-a|<R\}$ is an open disc containing $K$.  
The growth behavior is determined by $\lambda$'s close to $K$, 
so this definition does not depend on the choice of $D(a,R)$. 
An equivalent way to write \eqref{eq:res-growth} is 
\[
\norm{(T-\lambda)^{-1}}\le C' \max\{\dist(\lambda,\K)^{-s}, 
\dist(\lambda,\K)^{-1}\}, 
\quad 
\lambda\in\C\sm \K. 
\]
When $\K=\siT$ it is said that $T$ has \emph{resolvent growth of order $s$}. 

Let $\Omega$ be a bounded open set containing $K$ and 
let $f\colon\Omega\to\C$ be an analytic function. 
This paper is devoted to the study of the relation between the resolvent growth of $T$ near $K$
and the resolvent growth of $f(T)$ near $f(K)$, the latter operator being defined by the Riesz--Dunford calculus. 
We will assume throughout this paper 
that $f$ is not constant on any connected component of $\Om$. 
Notice also that $\Omega$ can be assumed as small as needed (containing $K$). 

The most relevant case in 
our results is when the compact set $K$ 
is precisely $\sigma(T)$. We work with a general $K$ because the proofs are 
just the same. 

\begin{theoremA}[Bakaev \cite{Bakaev98}]
\label{Bakaev_thm}
	If $T$ has resolvent growth of order $s$ near $\K$, then $f(T)$ also has resolvent growth of order $s$ near $\fK$.  
\end{theoremA}

Bakaev formulated this theorem for rational functions $f$.  
We will give a complete proof for analytic functions.

By a \emph{Lipschitz curve} in $\C$ we mean a curve 
that can be parametrized as $z=\zeta\cdot(t+if(t))$, 
$t\in I$, where $I\ss\R$ is a finite interval,  $f\colon I\to\R$ a Lipschitz function and 
$\zeta$ a complex constant of modulus $1$. 

For a set $G\subset \R^2$ and $\si>0$, we put
	\[
	[G]_\si:=\{y\in \R^2\sm G: \; 0<\dist(y,G)\le \si\}. \quad 
	\]
	Let $p\in (0,2)$ be a real number and let $G\subset \C$ be a compact set.  
	We say that $G$ is \emph{$p$-admissible} if there is a constant $C$ such that for any $a\in \R^2$ and any $R, \si>0$, 
	\begin{equation}  
		\label{eq:p-admiss}
		m\big([G]_\si\cap D(a,R)\big)\le C\si^{2-p} R^p. 
	\end{equation}
It is clear that each Lipschitz curve is $1$-admissible, but not vice versa. If the  
Assouad dimension of a compact set is less than two, then this set is $p$-admissible for some $p<2$. We refer to \cite{BelloYak-curves} for details. 

Here we will prove the following partial converse of Theorem~A. 

\begin{theorem}
	\label{converse_thm}
Let $T$ be a Banach space operator.  
Suppose $f(T)$ has resolvent growth of order $s$ near $\fK$, where $\K\supset \siT$, and that 
$f'$ does not vanish on $\K$. 

\begin{itemize}
	
	\item[(i)] If $K$ is contained 
in a finite union of Lipschitz curves, then $T$ has resolvent growth of order $s$ near $K$. 

    	\item[(ii)] If $K$ is $p$-admissible for some $p\in(0,2)$, then for any $\si>s$, $T$ has $\si$th order growth near $K$. 
\end{itemize}    	
\end{theorem}

	The proofs of Theorem~A and Theorem~\ref{converse_thm} will be given in 
	Section~\ref{sec:3}. Section~\ref{sec:general} contains some preliminary lemmas.

The statement of Theorem~\ref{converse_thm} is false if $K=\siT$ and  $f'$ has zeros on $\si(T)$ 
(just take as $T$ a $2\times 2$ Jordan block, whose diagonal element is a zero of $f'$). We will see in Example~\ref{ex:dense-sp} that the statement of Theorem~\ref{converse_thm} 
may fail if $K=\siT$ and $\siT$ 
is not $p$-admissible. 

Theorem~\ref{converse_thm} will rely on the following fact, which will be deduced from the results of~\cite{BelloYak-curves}.  

\begin{Lemma}
	\label{Lemma-resolv-curves}
	Let $T$ be a Banach space operator and let $K, \wt K\subset\C$ be  compact sets such that $\sigma(T)\subset K\subset \wt K$ 
	and 
	\begin{equation}
		\label{eq:res-growth-K}
		\norm{(T-\lambda)^{-1}}\le
		C \dist(\lambda, \wt K)^{-s}, \quad  
		\lambda\in D(b,R)\sm \wt K,  
	\end{equation}
	where $D(b,R)$ is an open disc containing $\wt K$. 

\begin{itemize}
	\item[(i)] If $\wt K$ is contained 
	in a finite union of Lipschitz curves, then $T$ has resolvent growth of order $s$ near $K$. 
	
	\item[(ii)] If $\wt K$ is $p$-admissible for some $p\in(0,2)$, then for any $\si>s$, $T$ has $\si$th order growth near $K$. 
\end{itemize}    	
\end{Lemma}

Notice that operators with the spectrum on a curve and power growth of the resolvent have been studied in many works, in particular, in a series of works by Kantorowitz, see, for instance,~\cite{Kan65}--\cite{Kan09}. 

The questions posed here are certainly interesting for unbounded operators. 
In particular, the case when $T$ is an (unbounded) generator of a $C_0$ group and one has to relate the resolvent growth of $T$ and that of the operators in the group 
$\{\exp(sT)\}$ for a single value of $s$ (or for all values $s\in\R$). We refer to 
Chapter 11 of the book \cite{Kantorovitz83} by S.~Kantorovitz for information on this problem in some special setting.

We recall that $T_0\in L(X)$ is called {\it scalar spectral} if there is a projection-valued spectral measure $E(\cdot)$, defined on all Borel subsets of $\C$, such that 
\[
T_0=\int_{\C} \la\, dE(\la). 
\]
In the Hilbert space case, this is equivalent to $T_0$ being similar to a normal operator.  
An operator $T\in L(X)$ is called {\it spectral} 
if it has a representation $T=T_0+N$, where 
$T_0$ is scalar spectral and $N$ is a quasinilpotent operator such 
that $NE(\de)=E(\de)N$ for any Borel set $\de\subset \C$. 
We refer to \cite{DunfSchw3} for the basics on spectral operators and to \cite{Rad91} for 
their simplified definition. 

Each scalar spectral operator has first order growth of the resolvent. If $T$ is scalar and its quasinilpotent part $N$
satisfies $N^{m+1}=0$ for some nonnegative integer $m$ (so that $N$ is nilpotent), then 
it is said that $T$ is a \emph{scalar operator of type $m$}. 
In this case, the resolvent of $T$ has $(m+1)$-th order growth. Therefore, the following result  can be seen as a counterpart of the above Theorem~\ref{converse_thm}. 

\begin{theoremB}[Apostol \cite{Apostol68}]
	\label{Apostol_thm}
	If $f(T)$ is (scalar) spectral and $f'$ does not vanish on $\si(T)$, then $T$ is (scalar) spectral. 
\end{theoremB}

	We refer to 
	the above-mentioned works by Kantorowitz and to 
	\cite{Sta72,DEY18} for relations between 
	the growth of the resolvent and spectrality of operators, whose spectrum lies on a curve. 

In Section~\ref{sec:Completeness}, we will study operators whose eigenspaces or root spaces form an unconditional basis. 
We discuss whether the property of unconditional basisness of eigenspaces or root spaces of $f(T)$ implies the corresponding property for $T$, and related issues. These results are motivated by Apostol's Theorem~B. We single out the notion 
of a ``good spectral operator'' and show its relevance. 

Completeness and unconditional basisness of root spaces of ordinary and partial differential operators is a classical, well-studied topic with enormous literature (we only cite the review \cite{Radzievskii}). But, surprisingly, not so many papers make an explicit use of its relationship with the theory of spectral operators by Dunford and Schwartz. 

One of our motivations was an application to 
operators, which have the closure of a finitely connected domain $\Om$ as a completely 
$K$-spectral set. A result of Benamara-Nikolski 
says that if the closed unit disc $\cl\D$ is a complete $K$-spectral set for $T$ and 
certain additional assumptions hold, then the operator $T$ is similar to a normal operator whenever it has first order resolvent growth. Given a finitely connected domain $\Om$ 
with analytic boundary, one can find a holomorphic function $f$ (defined on a neighbourhood of $\cl\Om$)
that maps $\Om$ to $\D$ and maps $\pt\Om$ to  
$\pt\D$. Then one can apply our results in order to generalize the Benamara-Nikolski theorem to 
operators $T$ that have $\cl\Om$ as a completely
$K$-spectral set. These results will be published elsewhere. 

The authors thank Yuri Tomilov for helpful suggestions.  

\section{Lemmas on analytic functions}
\label{sec:general}

Fix, as in the Introduction, a bounded linear operator $T$ on a Banach space $X$, a bounded open set $\Om$ containing $\sigma(T)$, an analytic function $f\colon\Om\to\C$ that is not constant in any connected component of $\Om$ and 
a compact set $K$ such that $\siT\ss K\ss \Om$. 

Then $f'$ has only finitely many zeros on $\K$, which we denote by 
$\xi_1, \dots, \xi_\ell$. 
We will assume that $f$ is bounded on $\Om$ and that $f'$ has no zeros on $\Om\sm\K$; this is achieved by substituting $\Om$ with a smaller open set $\Om'$ such that $\K\ss\Om'\ss\Om$. 
The orders  $m_j$ of the branching points $\xi_j$ of the mapping $f$ satisfy $m_j\ge 2$. By basic complex analysis, 
there is a small $\eps>0$ and neighbourhoods $W_j$ of $\xi_j$ 
with disjoint closures contained in $\Om$ such that 
$f$ maps $m_j$-to-one $W_j\sm \{\xi_j\}$ to the punctured discs $0<|z-f(\xi_j)|<\eps$. 
We may also assume that 
the discs $D(f(\xi_j),\eps)$,  
$D(f(\xi_k),\eps)$ are disjoint whenever $f(\xi_j)\ne f(\xi_k)$ 
and that 
\begin{equation}\label{eq:2diam}
\cl D(\xi_j, 2\diam W_j) \subset \Om. 
\end{equation}

We put $\Ombr=\bigcup_j W_j$.  
Consider the compact set $(\K\sm\Ombr)\cup \pt\Ombr$. 
Each $\la$ in $(\K\sm\Ombr)\cup \pt\Ombr$ has a neighborhood $U$ in $\Om$ such that 
$\cl U\subset \Om$, $f$ is one-to-one in $U$, 
$f'\ne 0$ in $\cl U$, the image $V=f(U)$ is open, and the function $g\colon V\to U$ defined by $f(g(z))=z$ is analytic on $V$. 
In this situation, we say that $g$ is an \emph{analytic branch} of $f^{-1}$. 
We will use the notation $U\leftrightarrow V$ to express that there exists an analytic branch of $f^{-1}$ from $V$ to $U$. 

For each $\la\in (\K\sm\Ombr)\cup \pt\Ombr$ consider a neighborhood $U=U_\la$ as above. 
Without loss of generality, assume that $U_\la$ is an open disc.  
Let $U_{\la_1},\dotsc,U_{\la_d}$ be a finite cover of $(\K\sm\Ombr)\cup \pt\Ombr$.  
Let $\Omega_1$ and $\Omega_2$ be open sets such that 
\begin{equation}
	\label{eq:finite-cover}
	(\K\sm\Ombr)\cup \pt\Ombr 
	\ss\Omega_2\ss\cl\Omega_2     \ss\Omega_1\ss\cl\Omega_1\ss U_{\la_1}\cup\dotsb\cup U_{\la_d}. 
\end{equation} 

We will also set
\[
\Omonebr= \Om_1\cup\Ombr,  \quad  \Omtwobr= \Om_2\cup\Ombr. 
\]
Notice that $K\subset\Omtwobr\subset\Omonebr$. 

The following result simply says that $f$ is Lipschitz on the compact set $\cl{\Om}_{1br}$, which follows from  well-known facts. We give the proof for the sake of completeness.

\begin{Lemma}\label{lem:Lipschitz}
There is a constant $C$ such that  
\[
\abs{f(\la)-f(\wt\la)}\le C\abs{\la-\wt\la} 
\] 
for every $\la,\wt\la\in\cl\Om_{1br}$. 
\end{Lemma}

\begin{proof}
Let $D_1,\dotsc,D_n$ be open discs such that 
\[
\cl\Om_{1br}\ss D_1\cup\dotsb\cup D_n\ss\cl{(D_1\cup\dotsb\cup D_n)}\ss\Omega, 
\]
and let $\delta>0$ be a Lebesgue number associated to this open covering, 
that is, if $A\subset \Omonebr$ has diameter $\diam A<\de$, then $A\subset D_j$ for some $j$. Take 
	$\la,\wt\la\in\cl\Om_{1br}$.

If $\abs{\la-\wt\la}\ge\delta$, then 
\[
\abs{f(\la)-f(\wt\la)}
\le
\diam f(\Omonebr)
\le
\frac{\diam f(\Omonebr)}{\delta}\abs{\la-\wt\la}
=:C_1\abs{\la-\wt\la}. 
\]
On the other hand, if $\abs{\la-\wt\la}<\delta$, then both $\la$ and $\wt\la$ belong 
to some disc $D_j$. By the convexity of discs, an application of the mean value theorem 
gives that 
$
\abs{f(\la)-f(\wt\la)}\le C_2\abs{\la-\wt\la}, 
$
where 
\[
C_2:=\max\{\abs{f'(s)}:s\in\cl{(D_1\cup\dotsb\cup D_n)}\}. 
\]
Therefore the statement follows by taking $C:=\max\{C_1,C_2\}$. 
\end{proof}

\begin{Lemma}\label{Lemma:dist}
	There is a positive number $\de$ such that 
	any two points $\la \ne \la'$ in $\Omonebr$ that satisfy  $f(\la)=f(\la')$ and 
	$|\la-\la'|<\de$ belong to the same set $W_j$, for some $j$. 
\end{Lemma}

\begin{proof} 
Let $\de$ be a Lebesgue number associated to the open covering 
\[
\Omonebr\ss U_{\la_1}\cup\dotsb\cup U_{\la_d}\cup W_1\cup\dotsb W_\ell. 
\]
Then, whenever $\la, \la'$ satisfy the above conditions, these points belong to the same set of the open covering. This set cannot be a set $U_{\la_j}$, because $f$ is injective on each of these sets. So both $\la$ and $\la'$ belong to a set $W_j$ for some $j$. 
\end{proof}

Notice that $f'\ne0$ on $\cl\Om_1$. 
We set
\begin{equation}\label{eq:tau}
	\tau:=\min\{\abs{f'(\lambda)}:\lambda\in\cl\Omega_1\}>0. 
\end{equation}

\begin{Lemma}\label{Lemma:rho}
	There is a positive number $\rho$ such that every $\la$ in $\cl\Omega_2$ has a neighborhood $Y_\la$ in $\Omega_1$ satisfying $Y_\la\leftrightarrow D(f(\la),\rho)$. 
\end{Lemma}

\begin{proof}
	Let $\delta_1>0$ be the Lebesgue number associated to the cover $\{U_{\la_j}\}$ of $\clOmone$. Take $\delta_2\in(0,\delta_1/2)$ sufficiently small so that the discs $D(\lambda,\delta_2)$ are contained in $\Omega_1$ for all $\lambda\in\cl\Omega_2$. 
	We claim that $\rho:=\tau\delta_2/4$ satisfies the statement. 
	Indeed, fix $\lambda\in\cl\Omega_2$. 
	Since $\diam D(\lambda,\delta_2)<\delta_1$, we have 
	$D(\lambda,\delta_2)\ss U_{\lambda_j}$ for some $j\in\{1,\dotsc,d\}$. By the Koebe $1/4$ theorem, $D(f(\lambda),\rho)\ss f(D(\lambda,\delta_2))$.  
	Therefore the set 
	\[
	Y_\lambda:=f^{-1}(D(f(\lambda),\rho))\cap D(\lambda,\delta_2)
	\] 
	fulfills the desired properties. 
\end{proof}

Given a set $S$ and functions $\alpha,\beta\colon S\to[0,\infty)$, we write 
\[
\alpha(s)\lesssim\beta(s), \quad s\in S
\]
if there exists a constant $C>0$, which does not depend on $s$, such that 
\[
\alpha(s)\le C\beta(s) \quad \tn{for all } s\in S. 
\]
If we have both $\alpha(s)\lesssim\beta(s)$ and $\beta(s)\lesssim\alpha(s)$, 
then we write 
\[
\alpha(s)\asymp\beta(s), \quad s\in S. 
\]
Put 
\[
\wtK:=f^{-1}\big(f(\K)\big)\cap\clOmonebr. 
\]
Clearly, $\wtK$ is compact, $\wtK\supset \K$ and $f(\wtK)=f(K)$. 

\begin{Lemma}\label{lem:bilipsch}
We have 
\[
\dist(\la, \wtK) \asymp \dist(f(\la),f(K)), 
\quad 
\la\in \cl{\Om}_2. 
\] 
\end{Lemma}

\begin{proof}
Let us first prove that 
\begin{equation}\label{eq:bilipsch1}
\dist(f(\la),f(K))\lesssim\dist(\la,\wtK),\quad\la\in \cl{\Om}_2. 
\end{equation}
For each $\lambda\in\cl{\Omega}_2$, let $\wt\lambda\in\wtK$ be such that 
$\dist(\lambda,\wtK)=\abs{\lambda-\wt\lambda}$. 
By Lemma~\ref{lem:Lipschitz},  
\[
\abs{f(\lambda)-f(\wt\lambda)}\lesssim\abs{\lambda-\wt\lambda}, 
\quad \lambda\in\cl\Omega_2. 
\]
Since the left hand side above is at least 
$\dist(f(\lambda),\fK)$, \eqref{eq:bilipsch1} follows. 

Now let us see that 
\begin{equation}\label{eq:bilipsch2}
\dist(\la,\wtK)\lesssim\dist(f(\la),\fK),\quad\la\in \cl{\Om}_2.
\end{equation}
Fix $\lambda\in\cl\Omega_2$ and let $\rho$ be as in Lemma~\ref{Lemma:rho}. 

If $\dist(f(\la),\fK)\ge\rho$ then
\[
\dist(\lambda,\wtK)\le\diam\Omtwobr\le(\rho^{-1}\diam\Omtwobr)\dist(f(\lambda),\fK). 
\]
Now suppose that $\dist(f(\la),\fK)<\rho$. 
Take a point $w\in \fK\cap D(f(\lambda),\rho)$ such that 
$\dist(f(\lambda),\fK)=\abs{f(\lambda)-w}$. 
Let $g\colon D(f(\lambda),\rho)\to Y_\lambda$ be an analytic branch of $f^{-1}$. 
Since $Y_\la\subset\Om_1$, 
we have $g(w)\in\wtK$. Observe that we also have $\abs{g'}\le\tau^{-1}$ in $D(f(\lambda),\rho)$, where $\tau$ is given by~\eqref{eq:tau}. Therefore 
\begin{align*}
\dist(\lambda,\wtK)
\le
\abs{g(f(\lambda))-g(w)}
\le
\tau^{-1}\abs{f(\lambda)-w}
=
\tau^{-1}\dist(f(\lambda),\fK). 
\end{align*}
This completes the proof. 
\end{proof}

For any $z\in\C$, we set $k(z)$ to be the number of preimages under $f$ of $z$ in $\Omonebr$, counted with their multiplicities. Notice that there is a uniform bound
\[
k(z)\le n_0, 
\] 
because the number of preimages of $z$ on each of the sets $U_{\la_j}, W_k$ is uniformly bounded. Let us denote the preimages of  $z$ in $\Omonebr$ by $\la_{z,1},\dotsc,\la_{z,k(z)}$. 

For $z\in f(\Omtwobr)$, define $\varphi_z\colon\Omonebr\to\C$ by 
\begin{equation}\label{eq:DefVarphi}
f(\la)-z=\varphi_z(\la)
(\la-\la_{z,1})\dotsb(\la-\la_{z,k(z)}). 
\end{equation}
This function is analytic on $\Omonebr$ and does not vanish there, hence $1/\varphi_z$ is also analytic on $\Omonebr$. 
In the next lemma, we are interested in its restriction to the domain $\Omtwobr$, which is smaller than 
$\Omonebr$.

\begin{Lemma}\label{lem:HinftyNorm}
We have 
\[
\norm{1/\varphi_z}_{H^\infty(\Omtwobr)}\lesssim 1, 
\quad 
z\in f(\Omtwobr).
\]
\end{Lemma}

\begin{proof}
Fix $z\in f(\Omtwobr)$.  
Let $\rho$ be the 
number from Lemma~\ref{Lemma:rho}. 

Let $\la$ be any point in $\Omega_2$ such that $f(\la)\neq z$.
If $\abs{f(\la)-z}\ge\rho$, then we have the bound 
\[
\Abs{\frac{1}{\varphi_z(\la)}}
=
\Abs{\frac{(\la-\la_{z,1})\dotsb
		(\la-\la_{z,k(z)})}{f(\la)-z}}
\le 
\frac{1+(\diam\Omega_1)^{n_0}}{\rho}\, . 
\]
If $\abs{f(\la)-z}<\rho$, then there is a neighborhood $Y_\la\ss\Om_1$ of $\la$ such that $Y_\la\leftrightarrow D(f(\la),\rho)$. 
In particular, for some $j$, the  
point $\la_{z,j}$ belongs to $Y_\la$. 
Let $g$ be the analytic branch of $f^{-1}$ from $D(f(\la),\rho)$ to $Y_\la$. 
We have 
\begin{align*}
\abs{\la-\la_{z,j}}
&=
\abs{g(f(\la))-g(z)}\\
&\le 
\Big[\sup_{t\in D(f(\la),\rho)}\abs{g'(t)}\Big]\abs{f(\la)-z}\\
&\le 
\tau^{-1}\abs{f(\la)-z}. 
\end{align*}
Therefore, by~\eqref{eq:DefVarphi}, in this case we obtain the bound 
\[
\Abs{\frac{1}{\varphi_z(\la)}}
\le 
\tau^{-1}\big(1+(\diam\Omega_1)^{n_0-1}\big).
\]
The above estimates show that the norms $\|1/\phi_z\|_{H^\infty(\Om_2)}$ are uniformly bounded 
when  $z$ ranges over $f(\cl\Om_{2br})$. 
Since $\pt\Ombr\subset\Om_2$ and $1/\phi_z$ is analytic in a 
neighbourhood of $\cl{\Om}_{2br}$, 
by the maximum principle $\|1/\phi_z\|_{H^\infty(\Om_2)}=\|1/\phi_z\|_{H^\infty(\Omtwobr)}$. 
This completes the proof. 
\end{proof}


\section{Proofs of Theorem A and Theorem \ref{converse_thm}}
\label{sec:3}

\begin{proof}[Proof of Bakaev's Theorem A]
It will suffice to estimate $\|(f(T)-z)^{-1}\|$ for $z\in f(\Omtwobr)\sm f(K)$. 
Since any such $z$ does not belong to the set $\{f(\xi_1),\dotsc,f(\xi_\ell)\}$, 
its preimages (under $f$) in $\Omonebr$ are all different. As before, 
we denote them by $\la_{z,1}, \dots, \la_{z,k(z)}$. Consider the decomposition into simple fractions 
\[
\frac{1}{(\la-\la_{z,1})\dotsb(\la-\la_{z,k(z)})}
=
\frac{c_{z,1}}{\la-\la_{z,1}}+\dotsb+
\frac{c_{z,k(z)}}{\la-\la_{z,k(z)}}. 
\]
By~\eqref{eq:DefVarphi}, 
\[
(f(\la)-z)^{-1}=
\varphi_z(\la)^{-1}\sum_{j=1}^{k(z)} c_{z,j}(\la-\la_{z,j})^{-1},  
\]
so applying the Riesz--Dunford calculus one gets  
\begin{equation}
\label{eq:fTz}
(f(T)-z)^{-1}=
\varphi_z(T)^{-1}\sum_{j=1}^{k(z)} c_{z,j}(T-\la_{z,j})^{-1}.
\end{equation}
By Lemma~\ref{lem:HinftyNorm}, 
the norms of $\varphi_z(T)^{-1}$
are uniformly bounded. 
Since $k(z) \le n_0$ uniformly in $z$,  
it will be sufficient to check the estimate for each term in the sum: 
\begin{equation}\label{eq:estim.each.sum}
\|c_{z,j}(T-\la_{z,j})^{-1}\|\lesssim \dist(z,\fK)^{-s}, \quad z\in f(\Omtwobr)\sm f(K). 
\end{equation}

So, fix an index $j$. 
Notice that 
\begin{equation}
\label{abs-czj}
|c_{z,j}|=\prod_{k\ne j} |\la_{z,j}-\la_{z,k}|^{-1}.  
\end{equation}
We distinguish two cases. 

\smallskip

{\scshape Case 1:} $\la_{z,j}\notin \Ombr$. 
Then Lemma~\ref{Lemma:dist} implies that 
$|c_{z,j}|\le \max\{1,\de^{-n_0+1}\}$. 
Therefore, invoking Lemma~\ref{lem:bilipsch} 
we obtain 
\begin{align*}
\|c_{z,j}(T-\la_{z,j})^{-1}\|
&\lesssim 
\dist(\la_{z,j},\K)^{-s}\\
&\lesssim 
\dist(\la_{z,j},\wtK)^{-s}\\
&\lesssim 
\dist(z,\fK)^{-s}, \quad  z\in f(\Omtwobr)\sm f(K).
\end{align*} 

\smallskip

{\scshape Case 2:} $\la_{z,j}\in \Ombr$. 
Hence $\la_{z,j}\in W_k$ for some $k$. 
We recall that 
$W_k$ contains a branching point $\xi_k$ of $f$ of order $m=m_k$. 
Here we use a part of Bakaev's arguments. 
If $\lambda_{z,p}\notin W_k$, then $\abs{\lambda_{z,j}-\lambda_{z,p}}\ge\delta$. 
On the other hand, if $\lambda_{z,p}\in W_k$ with $p\ne j$ (there are exactly $m-1$ of these $\lambda_{z,p}$), then
\[
\abs{\la_{z,j}-\la_{z,p}}\asymp\abs{\la_{z,j}-\xi_k}, \quad z\in f(\Omtwobr)\sm f(K)
\] 
Therefore 
\[
|c_{z,j}|\asymp |\la_{z,j}-\xi_k|^{-m+1}, \quad z\in f(\Omtwobr)\sm f(K). 
\]
Next, let $\zeta_{z,j}$ be a point of $\K$ such that $\dist(\lambda_{z,j},K)=\abs{\lambda_{z,j}-\zeta_{z,j}}$. 
Since $\xi_k\in \K$, we have  
$|\zeta_{z,j}-\xi_k|\le 2 |\la_{z,j}-\xi_k|\le 2\diam W_j$. 
Recall that, by ~\eqref{eq:2diam}, $\cl D(\xi_k, 2\diam W_k)\ss\Om$. 
Thus we can estimate 
\begin{align*}
\dist(z,\fK)
&\le 
|f(\la_{z,j})-f(\zeta_{z,j})|  \\
&\le 
|\la_{z,j}-\zeta_{z,j}|\cdot \max_{\cl D(\xi_k, 2\diam W_k)}\,|f'| \\
&\lesssim 
\dist(\la_{z,j}, \K)\cdot |\la_{z,j}-\xi_k|^{m-1}, \quad z\in f(\Omtwobr)\sm f(K). 
\end{align*}
Hence, using that $s\ge 1$ and $m\ge 2$, we have 
\begin{align*}
\|c_{z,j}(T-\la_{z,j})^{-1}\|
&\lesssim 
|\la_{z,j}-\xi_k|^{-m+1}\dist(\la_{z,j}, \K)^{-s}  \\
&\lesssim 
|\la_{z,j}-\xi_k|^{-m+1}\dist(z, \fK)^{-s}|\la_{z,j}-\xi_k|^{s(m-1)} \\
&\lesssim 
\dist(z,\fK)^{-s}, \quad z\in f(\Omtwobr)\sm f(K). 
\end{align*}
Therefore \eqref{eq:estim.each.sum} holds, as we wanted to prove. 
\end{proof}

\begin{Lemma}\label{lem:BoundResolvent}
	If $z\in f(\Omega_2)\sm\sigma(f(T))$ and $z\notin f(\Ombr)$, then 
	\[
	\norm{(f(T)-z)^{-1}}\asymp\max\{\norm{(T-\la)^{-1}}:\la\in f^{-1}(z)\cap \Om_1\}. 
	\]
\end{Lemma}

\begin{proof}
The inequality 
\[
\norm{(f(T)-z)^{-1}}\lesssim
\max\{\norm{(T-\la)^{-1}}:\la\in f^{-1}(z)\cap \Om_1\}
\]
follows from identities~\eqref{eq:fTz} and ~\eqref{abs-czj}.  
Indeed, since $z\notin f(\Ombr)$, for each $\la=\la_{z,j}$ satisfying $\la\in f^{-1}(\{z\})\cap \Om_1$, Case 1 in the previous proof will take place. 
Therefore in~\eqref{eq:fTz}, $c_{z,j}$ are uniformly bounded. 

To prove the reverse inequality, take any  $\lambda\in\Omega_1$ and consider the function $\psi_\lambda\colon\Omega_2\to\C$, given by $\psi_\la(t)=(f(t)-f(\lambda))/(t-\la)$. 
It is easy to see that the functions $\psi_\la$ are analytic on $\Om_2$ and their norms in $H^\infty(\Om_2)$ are uniformly bounded by a finite constant. 
By the same argument as in Lemma~\ref{lem:HinftyNorm}, 
the norms of $\psi_\la$ in $H^\infty(\Omtwobr)$ are uniformly bounded by the same constant. 
Therefore, given $z\in f(\Omega_2)\sm\sigma(f(T))$ and 
$\la\in f^{-1}(z)\cap \Om_1$, the identity 
\[
(T-\la)^{-1}=\psi_\la(T)\big(f(T)-z\big)^{-1},\]
readily gives  
\[
\norm{(T-\la)^{-1}}\le C \norm{(f(T)-z)^{-1}},   
\]
as we wanted to prove. 
\end{proof}

\begin{proof}[Proof of Lemma \ref{Lemma-resolv-curves}]
	It is well-known that the function 
	\[
	u(\la):=
	\log \|(T-\la)^{-1}\|
	\] 
	is subharmonic on $\C\sm \siT$. 
	Put $A=\cl{(\wt K\sm K)}$ and $B=K$, then 
	$\wt K=A\cup B$. We apply the results of~\cite{BelloYak-curves}. 
	
	Under the hypothesis of (i), $A$ is contained 
	in a finite union of Lipschitz curves, so that 
	inequality~\eqref{eq:res-growth-K} and 
	~\cite[part (1) of Corollary 2.8]{BelloYak-curves}
	imply that $T$ has resolvent growth of order $s$ near $B=K$. 
	
	Under the hypothesis of (ii), $A$ is $p$-admissible for some $p<2$. Put $g(t)=s\log(1+1/t)$. 
	By~\eqref{eq:res-growth-K}, there is a constant $C_1$ such that $u(z)-C_1\le g(\dist(z,\wt K))$ for 
	$z\in D(b,R)\sm \wt K$. 
	Fix any $\si>s$. 
	By ~\cite[Theorem 2.5]{BelloYak-curves}, we get that 
	for any $a>1$ there is $v\in (0,1)$
	such that 
	$u(z)-C_1\le ag(v\dist(z, K))$
	for $z\in D(b,R)\sm K$. By applying it to 
	$a=\si/s$, we get that $T$ has $\si$th order growth near $K$.  
\end{proof}

\begin{proof}[Proof of Theorem~\ref{converse_thm}]
        First we observe that since $f'$ does not vanish on $K$, $\Ombr=\emptyset$ and $\Omtwobr=\Om_2$. Suppose $f(T)$ has resolvent growth of order $s$ near $K$. Then, by Lemmas~\ref{lem:BoundResolvent} and ~\ref{lem:bilipsch}, 
		\begin{align*}
			\norm{(T-\la)^{-1}}
			&\lesssim
			\norm{(f(T)-  f(\lambda))^{-1}}\\
			&\lesssim
			\dist(f(\la),\fK)^{-s}\\
			&\lesssim
			\dist(\la,\wtK)^{-s}, \quad  \la\in \Om_2. 
		\end{align*}
	
		First assume that (i) holds, that is, $K$ is contained in a finite union of Lipschitz curves. Then it is easy to see that $\wtK$ also is contained in a finite union of Lipschitz curves. 
		By applying part (i) of  Lemma~\ref{Lemma-resolv-curves}, we get that 
		\[
		\norm{(T-\la)^{-1}}
		\lesssim
		\dist(\la, \K)^{-s}, \quad \la\in\Om_2. 
		\]
		
		Now assume that (ii) holds, that is, $K$ is $p$-admissible for some $p\in (0,2)$. Then it is easy to see that $f(K)$ and hence $\wtK$ also are $p$-admissible. 
		By applying part (ii) of  Lemma~\ref{Lemma-resolv-curves}, now we get that for any $\si>s$, 
\[
\norm{(T-\la)^{-1}}
\lesssim
\dist(\la, \K)^{-\si}, \quad \la\in\Om_2. \qedhere 
\]		
	\end{proof}

The following example shows that the geometric requirements on $\siT$ in Theorem~\ref{converse_thm} cannot be dropped.

\begin{example}\label{ex:dense-sp}
	Let $N$ be a compact normal operator on a Hilbert space, whose spectrum is contained in $\cl{D(0,1)}$ 
	and is so dense that 
	\[
	\dist(\la,\si(N))\le |\la|^3, \quad |\la|\le 1. 
	\]
	Let $T=N\oplus J$, where $J$ is the $2\times 2$ Jordan cell corresponding to eigenvalue $3$. Put $f(z)=z^2-3z$. 
	Then $f'\ne 0$ on $\si(T)$ and $T$ fails to have first order resolvent growth near $\si(J)=\{3\}$. However, $f(T)$ has first order resolvent growth. 
Indeed, 
\begin{align*}
\|(f(T)-zI)^{-1}\|
&=
\max\big(\dist(f(N),z)^{-1}, \|(f(J)-z)\|^{-1}\big)   \\
&\lesssim 
\dist(f(N),z)^{-1} \\
&\le 
\dist(f(T),z)^{-1}, \quad  |z|\le 5. 
\end{align*}
	Indeed, it is easy to see that 
	\[\dist(z,\si(N))\lesssim |z|^3, \quad |z|\le 1\] and that 
	$f(J)$ is a $2\times 2$ Jordan cell corresponding to eigenvalue $0$, so has second order resolvent growth.  
\end{example}

\section{Completeness and unconditional bases}\label{sec:Completeness}
Let $T$ be a bounded linear operator on a Banach space $X$. 
In this section, we assume $X$ to be separable. 
Let $\Omega$ be a bounded open set containing $\sigma(T)$. 
Let $f\colon\Omega\to\C$ be as in Section~\ref{sec:general}, that is, an analytic function, not constant on every connected component of $\Om$. 

In this section we study some relations between the root spaces and eigenspaces of $T$ with the corresponding spaces for $f(T)$. For example, we prove that $T$ is complete if and only if $f(T)$ is complete (see Theorem~\ref{thm:T-f(T)-complete}), and that under the condition $f'\ne0$ on $\siT$, the eigenspaces of $f(T)$ form an unconditional basis whenever the eigenspaces of $T$ do (see Theorem~\ref{thm-uncond-bases-eigensp}).

Let us begin by recalling standard definitions that we will employ. 

\begin{definition}
Given $\lambda\in\C$, the \emph{eigenspace} and the \emph{root space} of $T$ associated to $\lambda$ are, respectively, 
\[
\cE_T(\la):=\ker(T-\la),
\quad
R_T(\la):=\operatorname{clos}\bigcup_{j\ge1}\ker(T-\la)^j. 
\]
\end{definition}

Notice that $\cE_T(\la)\ne 0$ iff $R_T(\la)\ne 0$ iff  $\lambda$ belongs to the point spectrum $\sigma_p(T)$. 
Also, by \cite[Theorem 10.33]{Rudin-funct-anal-book}, 
under our conditions on $f$ we have $\si_p(f(T))=f(\si_p(T))$.

We denote by $\cE(T)$ the closed linear span of all eigenspaces of $T$, and 
by $R(T)$ the closed linear span of all root spaces of $T$. If $R(T)=X$ then $T$ is said to be \emph{complete}. 
The \emph{ascend} of $\la$ respect to $T$, that we denote by $a_T(\la)$, is the smallest positive integer $a$ (if exists) such that $R_T(\la)=\ker(T-\la)^a$; otherwise we set $a_T(\la)=\infty$.

The following statement is probably known. We include it as a starting point. 
\begin{theorem}\label{thm:T-f(T)-complete}
Let $T$ be a bounded linear operator on a Banach space $X$, 
$\Omega$ a bounded open set containing $\sigma(T)$ and let 
$f\colon\Omega\to\C$ be as in Section~\ref{sec:general}. 
Then $R(T)=R(f(T))$. 
In particular, $T$ is complete iff $f(T)$ is complete. 
If $f'\ne 0$ on $\siT$, then one also has 
$\cE(T)=\cE(f(T))$.
\end{theorem}

\begin{definition}
	A countable family $\{X_\alpha\}_{\alpha\in A}$ of linearly independent subspaces of $X$ 
	is said to be an \emph{unconditional basis} 
	in $X$ if it is complete in $X$ and there is a constant $C>0$ such that 
	\begin{equation}\label{eq:DefiUncond}
		\Norm{\sum_{\alpha\in F_1}x_\alpha}\le C\Norm{\sum_{\alpha\in F_2}x_\alpha}
	\end{equation}
	for all finite subsets $F_1,F_2$ of $A$ with $F_1\ss F_2$ and all $x_\alpha\in X_\alpha$. 
\end{definition}

For each basis $\{X_\al\}_{\al\in A}$ of subspaces of $X$ and each index $\al\in A$, the corresponding parallel projection 
$P_\al$ from $X$ onto $X_\al$ is defined, so that 
$P_\al P_\be=\delta_{\al \be}P_\al$. 
A basis $\{X_\al\}$ is unconditional if and only if 
the norms of the sums $\sum_{\al\in F}P_\al$, where $F$ ranges over all finite subsets of $A$, are uniformly bounded. 
For Hilbert spaces, the concepts of unconditional bases 
and \emph{Riesz bases} coincide.  
	We refer to \cite{AlbiacKaltonbook2006} and to \cite[Lecture VI]{NikolskiTreatise} 
	(especially, Concluding Remarks)
	for an exposition of these topics. 

Suppose that the index set $A$ is partitioned into 
disjoint finite sets $A_n$, $n\in \N$. If 
$Y_n=\sum_{\alpha\in A_n} X_\alpha$, then we say that 
the family of spaces $\{Y_n\}$ is obtained from 
$\{X_\alpha\}$ by \emph{packing}. It is clear that whenever 
$\{X_\alpha\}$ form an unconditional basis, 
$\{Y_n\}$ also are an unconditional basis in $X$. 

In what follows, when we say that the eigenspaces of $T$ (or the root spaces) form an unconditional basis we understand that they are countably many, that is, that $\si_p(T)$ is countable.

\begin{theorem}\label{thm-uncond-bases-eigensp}
	\begin{enumerate}
		\item If the eigenspaces of $T$ form an unconditional basis, the eigenspaces of $f(T)$ also form an unconditional basis. 
		
		\item If $f'\ne 0$ on $\si(T)$ and the eigenspaces of $f(T)$ form an unconditional basis, then the eigenspaces of $T$ form an unconditional basis. 
	\end{enumerate}
\end{theorem}

Now we will formulate our results on unconditional bases of root spaces. 

We start with the following observation. Suppose the root spaces of $T$ form an unconditional basis in $X$. Then we can associate to $T$ the operator $S$ such that 
\begin{equation}
\label{eq:def-S}
Sx=\la x \quad \text{for all }\la\in\si_p(T),\, x\in R_T(\la). 
\end{equation}
Since the spaces $\{R_T(\la):\la\in \si_p(T)\}$, 
form an unconditional basis, 
$S$ extends in a unique way as a bounded operator on $X$. Next we set $N=T-S$ (see \cite[Lecture VI]{NikolskiTreatise}). 
It is clear (by testing on root vectors) that $T$ and $S$ commute; therefore $N$ commutes with both $S$ and $T$. 
Every root space $R_T(\la)$ of $T$ is invariant under $S$ and under $N$. 

We give the following definition (motivated by the notion of spectral operators): 

\begin{definition}
We say that $T$ is a \emph{good discrete operator} if 
the point spectrum of $T$ is countable, 
its root spaces form an unconditional basis
and the corresponding operator $N$ is quasinilpotent. 
\end{definition}

\begin{theorem}\label{thm-uncond-bases-root-spaces}
	\begin{enumerate}
		\item If the root spaces of $T$ form an unconditional basis, then the root spaces of $f(T)$ form an unconditional basis.  
		
		\item Suppose $f'\ne 0$ on $\si(T)$. If  
		$f(T)$ is a good discrete operator, 
		then $T$ is a good discrete operator, in particular, 
		its root spaces form an unconditional basis. 
	\end{enumerate}
\end{theorem}

Let us make some observations on the scope of the notion of a good discrete operator. 
An alternative characterization of this class is as follows. 
\begin{proposition}
	\label{prop:characzn-good-discr-ops}
	An operator $T$ on $X$ is a good discrete operator if 
	and only if it is a spectral operator whose spectral measure 
	$E$ has countable support and satisfies $E(\{\la\})X=R_T(\la)$ for all $\la\in\mathbb{C}$. 
\end{proposition}
We also have the following 

\begin{proposition}\label{prop:finite-ascents}
If the root spaces of $T$ form an unconditional basis and 
the ascents $a_T(\la)$, $\la\in\si_p(T)$, are uniformly bounded, then $T$ is a good discrete operator. 
\end{proposition}

In particular, if the root spaces of $f(T)$ form 
an unconditional basis and the ascends of eigenvalues of $f(T)$ are uniformly bounded, then part (2) of Theorem~\ref{thm-uncond-bases-root-spaces} is applicable, so that the root spaces of $T$ form an unconditional basis. 
It is worth to remind here 
~\cite[Corollary XV.5.7]{DunfSchw3}, which says that whenever 
$T$ is a spectral operator of type $m$, $f(T)$ is also 
a spectral operator of type $m$. 

The unconditional basisness of the root spaces of $f(T)$ itself does not imply the unconditional basisness of the root spaces of $T$, as the following result shows.

\begin{theorem}\label{thm:counterexample-model-opers}
There exists an operator $T$ on a Hilbert space $H$ and a holomorphic function 
$f:\Om\to \C$ such that $\siT\subset \Om$ and $f'\ne 0$ on 
$\Om$ such that the root spaces of $f(T)$ form an 
unconditional (Riesz) basis, whereas the root spaces of $T$ 
are not a Riesz basis of $H$. 
\end{theorem}

\begin{remark}
If eigenspaces of $T$ form an unconditional basis, then, as it is easy to see, the spectrum of operator $T$ equals to the closure of its point spectrum. However, 
if the root spaces of $T$ form an unconditional basis, then 
the spectrum $\si(T)$ can be much larger than the closure of 
$\si_p(T)$. For instance, let $J_n(\la)$ denote the canonical 
$n\times n$ Jordan block, corresponding to an eigenvalue $\la$. Consider the operator 
\[
T=\oplus_{n\in \N}J_n\Big(\frac 1n\Big), 
\]
acting on $\ell^2$ (we mean here an orthogonal sum). 
Its point spectrum consists of points $1/n$, $n\in \N$. 
As it is easy to see, 
the root spaces of $T$ form an unconditional basis of subspaces of $\ell^2$ (in fact, an orthogonal basis). 
Easy estimates of the norms of inverses of Jordan blocks imply that the spectrum of $T$ coincides with the closed unit disc $\cl\D$. 
\end{remark}

One can characterize good spectral operators under a little bit stronger requirement than that of unconditional basis. 

\begin{definition}
	We say that a family $\{X_\alpha\}_{\alpha\in A}$ of linearly independent subspaces of $X$ 
	is a \emph{strongly unconditional basis} in its closed linear span if there is a constant $C'>0$ such that whenever 
	$F\ss A$ is finite and $x_\alpha, 
	x'_\al\in X_\alpha$ and 
	$\norm{x'_\al}\le C\norm{x_\al}$ for 
	$\al\in F$, one has 
	\begin{equation*}
		\Norm{\sum_{\alpha\in F}x'_\alpha}\le C'C\Norm{\sum_{\alpha\in F}x_\alpha}. 
	\end{equation*}
\end{definition}
Obviously, if $\{X_\alpha\}_{\alpha\in A}$ is a strongly unconditional basis then it is an unconditional basis. 
We remark that this property has appeared in Dunford and Schwartz \cite[Lemma XV.6.8]{DunfSchw3}. In the case of a Hilbert space, a basis is strongly unconditional iff it is unconditional iff it is a Riesz basis. If all spaces $X_\al$ 
are one-dimensional and form unconditional basis, then it is automatically strongly unconditional; this follows from \cite[Proposition 3.1.3]{AlbiacKaltonbook2006}.
If one takes $X=\ell^p\oplus \ell^q$ for $1<p< q<\infty$ 
and $x_n=e_n\oplus 0$, 
$x'_n=0 \oplus e_n$, 
(where $\{e_n\}$ denotes the standard basis both in $\ell^p$ and $\ell^q$), then the two-dimensional spaces $X_n=\operatorname{span} (x_n, x'_n)$ are an unconditional basis, which is not strongly unconditional. 

\begin{theorem}\label{thm:charactzn-good-discrete}
Suppose the root spaces of $T$ form a strong unconditional basis. Define 
$N$ as before, and put $N_\la=N|_{R_T(\la)}$ for $\la\in\si_p(T)$. 
Then $T$ is a good discrete operator if and only if 
\begin{equation}\label{eq:N^k}
\sup_{\la\in\si_p(T)} \|N_\la^k\|^{1/k}\to 0
   \quad \text{as }k\to\infty. 
\end{equation}
\end{theorem}

Let us pass to the proofs. We start with some preparations to prove Theorem \ref{thm-uncond-bases-eigensp}. 

First we reproduce the following result. 

\begin{Lemma}[Dunford and Schwartz \cite{DunfSchw3}, Theorem XV.8.2]
	\label{lem:DS-XV.8.2}
	Suppose $T$ is a spectral operator and $E$ is its spectral measure. 
	Then for every vector $x\in X$ and $k\ge 0$ we have 
	$(T-\la)^kx=0$ if and only if $E(\{\la\})x=x$ and $N^kx=0$. 
\end{Lemma}

It follows that, whenever $T$ is spectral and $E$ is its spectral measure, $R_T(\la)\subset E(\{\la\})$ for all $\la\in\mathbb{C}$. 

An unconditional basis $\{X_\la\}_{\la\in \Lambda}$ in $X$, 
where $\Lambda$ is a countable bounded subset of $\C$, 
gives rise to a spectral measure $E$ in $X$, defined on a dense subset of $X$ by 
\begin{equation}\label{eq:star}
E(\de)\sum_{\la\in F}x_\la=\sum_{\la\in \de\cap F}x_\la, 
\quad \de\ss\C\text{ a Borel set}, 
\end{equation}
where $F\ss \Lambda$ is finite and $x_\la\in X_\la$ for 
$\la\in \Lambda$. 
The condition~\eqref{eq:DefiUncond} ensures that 
$E(\de)$ extends by continuity to $X$ and that it is indeed a 
spectral measure in $X$. It is easy to see that every spectral measure $E$, whose support is a bounded countable subset of $\C$, is of this form. 

We will use the following fact. 

\begin{Lemma}\label{lem:uncond-basis-count-supp}
Let $T$ be a bounded linear operator on $X$. Then the following assertions are equivalent. 
\begin{enumerate}

\item $\si_p(T)$ is countable and 
the eigenspaces of $T$ form an unconditional basis in $X$; 

\item $T=\int \la\, dE(\la)$, where $E$ is a spectral measure with countable support. 
\end{enumerate}
Moreover, if (2) holds, then the support of $E$ coincides with $\si_p(T)$. 
\end{Lemma}

\begin{proof}
To prove that (1) implies (2), we define $E$, using the unconditional basis $\{\ker(T-\la): \la\in\si_p(T)\}$, 
by formula~\eqref{eq:star}. Then for any $\la\in\si_p(T)$ and any $x\in \ker(T-\la)$, $Tx=\big(\int \la\, dE(\la)\big)x$, which implies (2). 

Now suppose (2), and let us see that it implies (1). 
First we observe that for any $\la$, by Lemma~\ref{lem:DS-XV.8.2}, $\ker(T-\la)=E(\{\la\})$. 
Therefore, if $S$ is the countable support of $E$, then  
$S=\si_p(T)$, so $\si_p(T)$ is countable. The fact that 
the eigenspaces of $T$ form an unconditional basis follows directly from the definition of a spectral measure. 
\end{proof}

\begin{proof}[Proof of Proposition \ref{prop:characzn-good-discr-ops}]
If $T$ is a good discrete operator, then, 
by the above Lemma~\ref{lem:uncond-basis-count-supp}, 
$S$ is spectral and its spectral measure $E$ has countable support. Hence $T$ is also spectral, with the same spectral measure $E$. Since $R_T(\la)\ss E(\{\la\})$ for all $\la$ and 
$R_T(\la)$, $\la\in\si_p(T)$, form an unconditional basis of $X$, we get that $R_T(\la)= E(\{\la\})$ for all $\la\in \si_p(T)$ (equivalently, for all $\la\in\mathbb{C}$). 

Conversely, assume that $T$ is a spectral operator whose spectral measure $E$ meets the declared conditions. Since 
the spaces $E(\{\la\})X$, $\la\in\si_p(T)$, form an unconditional basis and $R_T(\la)=E(\{\la\})X$ for all $\la$, 
we get that the root spaces of $T$ form an unconditional basis. Hence \eqref{eq:def-S} defines an operator $S$, which is a scalar part of $T$. The corresponding operator $N=T-S$ is quasinilpotent, and we conclude that $T$ is a good discrete operator. 
\end{proof}

We recall that, whenever $T$ is a spectral operator 
with spectral measure $E_T$, 
$f(T)$ is also spectral operator, whose spectral measure is given by  
\begin{equation}\label{eq:change-sp-measure}
E_{f(T)}(\de)=E_T(f^{-1}(\de))  
\end{equation}
for all Borel subsets $\de$ of $\C$. 
See Dunford-Schwartz \cite[Theorem XV.5.6]{DunfSchw3}. 

\begin{proof}[Proof of Theorem \ref{thm-uncond-bases-eigensp}]
	 (1) Suppose the eigenspaces of $T$ form an unconditional basis. Then, by Lemma~\ref{lem:uncond-basis-count-supp}, 
$T=\int \la\, dE(\la)$, where $E$ is a spectral measure with a countable support $S$. If we define a spectral measure 
$E_{f(T)}$ on $X$ by~\eqref{eq:change-sp-measure}, 
then 
\[
f(T)=\int f(\la)\, dE(\la)=\int \mu\, dE_{f(T)}(\mu), 
\]
by the ``change of measure principle'' (see Section X.1 in 
Dunford and Schwartz \cite{DunfSchw2}). Since $E_{f(T)}$ has a countable support, it follows by Lemma~\ref{lem:uncond-basis-count-supp} that 
eigenspaces of $f(T)$ form an unconditional basis.

(2) Suppose now that $f'\ne 0$ and that eigenspaces 
of $f(T)$ form an unconditional basis. Hence 
$f(T)$ is a scalar spectral operator: 
\[
f(T)=\int \mu\, dE_{f(T)}(\mu), 
\]
where the support of $E_{f(T)}$ is countable. 
By Apostol's Theorem~B (see the Introduction), $T$ is scalar spectral, with the corresponding spectral measure $E$. 
Then $E_{f(T)}$ and $E$ are related by ~\eqref{eq:change-sp-measure}, and it follows that the support of $E$ is countable. 
Once again, by Lemma~\ref{lem:uncond-basis-count-supp}, 
eigenspaces of $T$ form an unconditional basis. 
\end{proof}

We formally put $K=\siT$ and elect $\Om_1$ and $\Om_2$ as in Section~\ref{sec:general}. 

We assume, for the sake of clarity in the notation, that 
the preimages $\lambda_{z,1},\dotsc,\lambda_{z,n(z,\Omtwobr)}$ 
of $z\in\C$ in $\Omega_1$ are labeled so that 
\[
\lambda_{z,1},\dotsc,\lambda_{z,k(z)}\in\sigma_p(T),
\quad 
k(z):=n(z,\sigma_p(T)). 
\]

\begin{proposition}\label{prop:RTandNT}
	Let $z\in\sigma_p(f(T))$. 
	\begin{enumerate}
		\item
		One has 
		\begin{equation*}\label{eq:RT}
			R_{f(T)}(z)=R_T(\la_{z,1})+\dotsb+R_T(\la_{z,k(z)}).
		\end{equation*}
		If the ascends $\{a_T(\la_{z,j})\}_{j=1}^{k(z)}$ are finite, then here we have direct sums.  
		\item
		If $f'\ne0$ on $\sigma_p(T)$, then 
		\[
		E_{f(T)}(z)=E_T(\la_{z,1})\dotplus\dotsb\dotplus E_T(\la_{z,k(z)}).
		\] 
	\end{enumerate}	
\end{proposition}

\begin{proof}
	Assume that $h\in \ker(f(T)-z)^a$, for some positive integer $a> 0$. 
	Analogously to \eqref{eq:DefVarphi}, now we have 
	 \[
		f(\la)-z=\Psi_z(\la)(\la-\la_{z,1})^{m_{z,1}}\dotsb(\la-\la_{z,n(z,\Omtwobr)})^{m_{z,n(z,\Omtwobr)}},  
	 \]
	where the function $\Psi_z$ is analytic on $\Omtwobr$ and does not vanish there. 
	Then
	\begin{equation}\label{eq:RT1}
		0=(f(T)-z)^a h=\Psi_z(T)^a(T-\la_{z,1})^a\dotsb(T-\la_{z,n(z,\Omtwobr)})^a h. 
	\end{equation}
	Since $\Psi_z(T)$ is invertible, 
	\[
	(T-\la_{z,1})^{am_{z,1}}\dotsb(T-\la_{z,n(z,\Omtwobr)})^{am_{z,n(z,\Omtwobr)}} h=0.
	\]
	The decomposition into simple fractions 
\[
		\prod_{i=1}^{n(z,\Omtwobr)} (\la-\la_{z,i})^{-am_{z,i}}
		=
		\frac{Q_{z,1}(\la)}{(\la-\la_{z,1})^{am_{z,1}}}
		+\dotsb+
		\frac{Q_{z,n(z,\Omtwobr)}(\la)}{(\la-\la_{z,n(z,\Omtwobr)})^{am_{z,n(z,\Omtwobr)}}}
\]
	yields the identity $h=h_1+\dotsb+h_{n(z,\Omtwobr)}$, where 
	\[
	h_i=Q_{z,i}(T)\prod_{j\neq i}(T-\la_{z,j})^{am_{z,j}} h.
	\]
	Notice that $(T-\la_{z,i})^{am_{z,i}} h_i=0$. Therefore, statement (1) follows using that 
	$\ker(T-\lambda_{z,i})=\{0\}$ for $i>k(z)$. 
	
	Taking $a=m_{z,1}=\dotsb=m_{z,n(z,\Omtwobr)}=1$ in the previous argument we obtain (2). 
	
	Finally, both statements about direct sums follow from basic linear algebra. 
\end{proof}

\begin{proof}[Proof of Theorem~\ref{thm:T-f(T)-complete}]
	Take $\lambda\in\sigma_p(T)$ and set $z:=f(\lambda)\in\sigma_p(f(T))$. 
	Using item (1) of Proposition~\ref{prop:RTandNT}, we have 
	\[
	R_T(\lambda)\ss R_{f(T)}(z)\ss R(f(T)). 
	\]
	Therefore $R(T)\ss R(f(T))$. 
	
	Now take any $z\in\sigma_p(f(T))$. Using again item (1) of Proposition~\ref{prop:RTandNT}, we have 
	\[
	R_{f(T)}(z)=R_T(\la_{z,1})+\dotsb+R_T(\la_{z,k(z)})\ss R(T).
	\]
	Hence $R(f(T))\ss R(T)$. Therefore $R(T)=R(f(T))$. 
	
	An analogous argument using item (2) of Proposition~\ref{prop:RTandNT} when $f'\ne0$ 
	on $\sigma_p(T)$ yields the equality $E(T)=E(f(T))$. 
\end{proof}

\begin{proof}[Proof of Theorem~\ref{thm-uncond-bases-root-spaces}]
(1) Suppose that the root spaces of $T$ form an unconditional basis. Then, by Proposition~\ref{prop:RTandNT},  
the family of root space of $f(T)$ is obtained from the family of root spaces of $T$ by packing. Hence the root spaces of $f(T)$ also are an unconditional basis in $X$. 

(2) Suppose now that $f(T)$ is a good discrete operator 
and $f'\ne0$ on $\siT$. By Proposition~\ref{prop:characzn-good-discr-ops}, 
the spectral measure $E_{f(T)}$ of $f(T)$  
has a countable support, say, $W$. 
By Apostol's Theorem B, $T$ is spectral. Denote by $E_T$ its spectral measure. Formula~\eqref{eq:change-sp-measure} 
implies that $E_T(f^{-1}(W))=I$, so that $E_T$ has a countable support. By Theorem~\ref{thm:T-f(T)-complete}, the root spaces of $T$ are complete. Since $R_T(\la)\ss E_T(\{\la\})X$ for any 
$\la\in\si_p(T)$ and spaces $\big\{E_T(\{\la\})X: 
\la\in\si_p(T)\big\}$ form an unconditional basis of $X$, it follows that $R_T(\la)= E_T(\{\la\})X$
for all $\la\in\si_p(T)$. Once again, by Proposition~\ref{prop:characzn-good-discr-ops}, 
$T$ is a good spectral operator. 
\end{proof}

\begin{proof}[Proof of Proposition \ref{prop:finite-ascents}]
	Suppose that the root spaces of $T$ form an unconditional basis and the supremum $m:=\sup \{a_T(\la): \la\in\si_p(T)\}$ is finite. 
	Define $S$ by~\eqref{eq:def-S} and set $N=T-S$. For any $\la\in\si_p(T)$ and any $x\in R_T(\la)$, 
	\[
	N^mx=(T-S)^m x=(T-\la)^mx =0. 
	\] 
	Since the roots spaces $R_T(\la)$ are complete, 
	$N^m=0$ on the whole space $X$. So $T=S+N$, where 
	$S$ is scalar spectral, $N$ is nilpotent and $S$ and $N$ commute. It follows that $T$ is spectral, so it is a good discrete operator. 
\end{proof}

We will need the following lemma. 
\begin{Lemma}\label{lem:angles-root-spaces-matrices}
Let $\la$ be a complex number such that $0<|\la|<1$. 
Then for any $n\in \N$, there exists a linear operator  
$M_n$ 
on a $2n$-dimensional complex Hilbert space $H_n$
such that $\si(M_n)=\{0,\la\}$, $\|M_n\|\le 1$ for all $n$, and the angles between the root spaces 
$R_{M_n}(0)$ and $R_{M_n}(\la)$ tend to $0$ as $n\to\infty$.  
\end{Lemma}

We recall that the angle between two subspaces of a Hilbert space is defined as the infimum of angles between nonzero vectors in these two spaces. 

In the proof of this lemma, we will use the classical Hardy spaces $H^p$ of analytic functions in the unit disc $\D$ ($1\le p\le\infty$). As usual, we identify an analytic function $f$ in $H^p$ with its boundary limit values, defined a.e. on $\pt\D$. With this identification, $H^p$ becomes a closed subspace of $L^p(\pt\D)$. 
The space $H^2$ is Hilbert. 
We set $H^2_-= L^2(\pt\D)\ominus H^2$ (here we mean the orthogonal complement). 

\begin{proof}[Proof of Lemma \ref{lem:angles-root-spaces-matrices}]
Set $b_\la(z)=(z-\la)/(1-\bar\la z)$ to be a Blaschke factor. 
Consider the finite Blaschke product 
$\Theta=b_0^n b_{\bar\la}^n$. 
The \emph{model space}~\cite{NikolskiTreatise}
\[
H_n:=K(\Theta)= H^2\ominus \Theta H^2 = \{x\in H^2:\overline\Theta x\in H^2_-\}
\]
is a $2n$-dimensional Hilbert space which is invariant under the 
backward shift $S^*f(z)=(f(z)-f(0))/z$ on $H^2$. 
We put 
\[
M_n=S^*|_{H_n}. 
\]
It it the adjoint to the \emph{model operator}. 
Since $\|S^*\|=1$, we have $\|M_n\|\le 1$ for all $n$. 
As it follows from the Liv\v sic--Moeller Theorem~\cite[Lecture III]{NikolskiTreatise}, the spectrum of $M_n$ consists of points $0$, $\la$, each of algebraic multiplicity $n$. It is easy to see that 
\[
R_{M_n}(0)=K(b_0^n), \quad 
R_{M_n}(\la)=K(b_{\bar\la}^n).  
\]
It remains to check that the angle 
between $K(b_0^n)$ and $K(b_{\bar\la}^n)$ goes to $0$ 
as $n\to\infty$. To see it, notice that the normalized reproducing kernel 
\[
k_{\bar\la}=\frac{(1-|\la|^2)^{1/2}}{1-\la z}
\]
belongs to $K(b_{\bar\la}^n)$ and 
satisfies $\|k_{\bar\la}\|=1$. 
The space $K(b_0^n)$ consists of polynomials in $z$ of 
degree less or equal than $n-1$. Hence 
\begin{multline*}
\|k_{\bar\la}-P_{K(b_0^n)} k_{\bar\la}\|=
(1-|\la|^2)^{1/2}\Big\|\sum_{k=n}^\infty \la^k z^k\Big\|_{H^2}      \\
=(1-|\la|^2)^{1/2} 
  \Big(\sum_{k=n}^\infty |\la|^{2k}
\Big)^{1/2}
=|\la|^n\to 0\quad \text{as }n\to\infty. 
\end{multline*}
Therefore the angle between 
$k_{\bar\la}\in K(b_{\bar\la}^n)$ and 
$P_{K(b_0^n)}k_{\bar\la}$ goes to $0$ as 
$n\to\infty$. 
\end{proof}

\begin{proof}[Proof of Theorem \ref{thm:counterexample-model-opers}]
	Set $\la=2e^{i\pi/7}-2$; then $|\la|<1$. 
	Let $\{M_n\}$ be the finite dimensional operators 
	supplied by the above Lemma~\ref{lem:angles-root-spaces-matrices}. They act on spaces $H_n$, $\dim H_n=2n$. Put 
	$\phi_n=1/(20n)$. We set 
	\[
	L_n=e^{-i\phi_n}(M_n+2), \quad T=\oplus_n L_n, 
	\]
	so that $T$ acts on the Hilbert space $H=\oplus_n H_n$. The (point) spectrum of $L_n$ consists of points $\mu_n=2e^{-i\phi_n}$ and $\nu_n=2e^{-i\phi_n+\pi i/7}$. 
	These points lie on the circle $|z|=2$ and are all distinct. We have 
	\[
	\si_p(T)=\{\mu_n:n\in \N\}\cup \{\nu_n:n\in \N\}. 
	\]
	Since $\|M_n\|\le 1$ for all $n$, it is easy to see that 
	the spectrum of $T$ is contained in the closed disc 
	$\{z: |z-2|\le 1\}$. (In fact, using 
	the two-sided estimates of the model operator, given in~\cite[Section VI, \S4, Thm. 4.1 and Proposition 4.2]{NagyFoiasbook-edition2}, one can check that $\siT$ coincides with this disc.) Now consider the function $f(z)=z^{14}$. It is analytic on $\siT$ and $f'\ne 0$ on $\siT$. Notice that $f(\mu_n)=f(\nu_n)$ for all $n$, but all numbers $f(\mu_n), n\in \N$, are distinct. We have 
	$\si_p(f(T))=\{f(\mu_n): n\in N\}$ and 
	$R_{f(T)}(f(\mu_n))=H_n$. Hence the root spaces 
	$R_{f(T)}(z)$, $z\in\si_p(f(T))$, form an unconditional (in fact, an ortonormal) basis in $H$. 
	On the other hand, for any $n$, 
	the root spaces $R_T(\mu_n)$, $R_T(\nu_n)$ 
	are contained in $H_n$ and the angle between them goes to $0$ as $n\to\infty$, by 	Lemma~\ref{lem:angles-root-spaces-matrices}. 
\end{proof}

\begin{proof}[Proof of Theorem~\ref{thm:charactzn-good-discrete}]
Suppose the root spaces of $T$ form a strong unconditional basis, and let $N$ be the operator associated to $T$. Then $T$ is a good discrete operator iff $N$ is quasinilpotent, that is, $\|N^k\|^{1/k}\to 0$ as $k\to\infty$.  
By the above property of strong unconditionality, 
\[
\|N^k\|\asymp \sup_{\la\in\si_p(T)}\|N_\la^k\|. 
\]
Hence $\lim_{k\to\infty}\|N^k\|^{1/k}= 0$ iff~\eqref{eq:N^k} holds. 
\end{proof}

\begin{remark}
	It is easy to see that Apostol's Theorem B is no longer valid if $f$ has a branching point on the spectrum of $T$. 
	One example could be $f(z)=z^2$ and 
	$T=
	\begin{psmallmatrix}
		0&0\\
		1&0
	\end{psmallmatrix}
	$,
	a non-trivial Jordan block. 
	There are other examples when $T$ has no eigenvalues. 
	For instance, define an operator $T$ on 
	the vector $L^2$ space $L^2([0,1],\C^2, d\mu)$ by 
	$Tf(x)=T(x)f(x)$, where $\mu$ is the 
	Lebesgue measure on $[0,1]$ and 
	$
	T(x)= 
	\begin{psmallmatrix}
		x & 0\\
		2 & -x
	\end{psmallmatrix}
	$. 
    It is easy to see that if $T$ is similar to a normal operator, 
	then this normal operator can be written as 
	$Sf(x)=S(x)f(x)$ on $L^2([0,1],\C^2, d\mu)$, where  
	$
S(x)= 
\begin{psmallmatrix}
	x & 0\\
	0 & -x
\end{psmallmatrix}
$ 
and the similarity transform $W$
should be of the form $Wf(x)=W(x)f(x)$ for a.e. $x$, where 
$W(x)$ is a $2\times 2$ matrix of $L^\infty$ functions.   
The formula $W(x)T(x)=S(x)W(x)$ implies that $W(x)$ should have the form 
	$
	W(x)=
	\begin{psmallmatrix}
		a(x)&0\\
		-b(x)&\, xb(x)
	\end{psmallmatrix}. 
	$
	Since 
	\[
	\norm{W}\norm{W^{-1}}=(\esssup_x\norm{W(x)})(\esssup_x\norm{W^{-1}(x)}) 
	\]
	and 
	\[
	\norm{W(x)}\norm{W^{-1}(x)}\ge\abs{a(x)}\abs{x^{-1}a^{-1}(x)}=\abs{x}^{-1},
	\]
	it follows that there is no bounded invertible transform $W$ of this type. 

One can also take $\mu$ to be a discrete measure, 
concentrated on $\{1/n: n\in \N\}$, in which case 
$T$ is a complete operator such that the point $0$, where $f'=0$, is not its eigenvalue. 
\end{remark}

\bibliographystyle{siam}

\bibliography{biblio_full}
\end{document}